\documentclass[12pt,oneside]{amsart}
\usepackage{amsmath}
\usepackage{amsthm}
\usepackage{amsfonts}
\usepackage{amssymb}
\usepackage{enumerate}
\newtheorem{theorem}{Theorem}
\newtheorem{lemma}[theorem]{Lemma}
\newtheorem{prop}[theorem]{Proposition}
\newtheorem{corollary}[theorem]{Corollary}

\theoremstyle{definition}

\theoremstyle{remark}
\newtheorem{remark}[theorem]{Remark}
\newcommand{\DD}{{\mathbb D}}

\newcommand{\CC}{{\mathbb C}}

\newcommand{\TT}{{\mathbb T}}

\DeclareMathOperator{\Aut}{Aut}

\renewcommand{\phi}{\varphi}


\begin{document}
%\author{\L ukasz Kosi\'nski, Pascal J. Thomas, W\l odzimierz Zwonek}

\address{Institute of Mathematics, Faculty of Mathematics and Computer Science, Jagiellonian
University,  \L ojasiewicza 6, 30-348 Krak\'ow, Poland}\author{\L ukasz Kosi\'nski}\email{lukasz.kosinski@im.uj.edu.pl}

\address{Universit\'e de Toulouse, UPS, INSA, UT1, UTM, Institut de Math\'ematiques de Toulouse, F-31062 Toulouse, France}\author{Pascal J. Thomas}\email{pascal.thomas@math.univ-toulouse.fr}

\address{Institute of Mathematics, Faculty of Mathematics and Computer Science, Jagiellonian
University,  \L ojasiewicza 6, 30-348 Krak\'ow, Poland}\author{W\l odzimierz Zwonek}\email{wlodzimierz.zwonek@im.uj.edu.pl}

\thanks{The authors were supported by the grant of the Polish National Centre no. UMO-2013/08/M/ST1/00986 promoting the cooperation between the groups of complex analysis in the Paul Sabatier University in Toulouse and the Jagiellonian University in Krak\'ow}

\title{Coman conjecture for the bidisc}
\maketitle
\maketitle
\begin{abstract}
In the paper we show the equality between the Lempert function and the Green function with two poles with equal weights in the bidisc thus giving the positive answer to a conjecture of Coman in the simplest unknown case. Actually, a slightly more general equality is proven which in some sense is natural when studied from the point of view of the Nevanlinna-Pick problem in the bidisc.
\end{abstract}

\section{Presentation of the problem and its history}
Let $D$ be a domain in $\mathbb C^n$ and let $\emptyset\neq P:=\{p_1,\ldots,p_N\}\subset D$ where $p_j\neq p_k$, $j\neq k$. Let also $\nu:P\to(0,\infty)$. Denote $\nu_j:=\nu(p_j)$. Let $z\in D$.

 Define $l_D(z;P;\nu):=l_D(z;(p_1,\nu_1),\ldots,(p_N,\nu_N))$ as the infimum of the numbers
$$
\sum\sb{j=1}\sp{N} \nu_j \log|\lambda_j|
$$
such that there is an analytic disc $\psi:\mathbb D\to D$ with $\psi(0)=z$, $\psi(\lambda_j)=p_j$, $j=1,\ldots,N$. 

Recall that $l_D(z;P;\nu)=\min\{l_D(z;A;\nu_{|A}):\emptyset \neq A\subset P\}$ (see \cite{Nik-Pfl 2006} for arbitrary $D$ or \cite{Wik 2001} for $D$ convex). The last equality will be of interest  for us since in the case of taut domains (convex and bounded domains are taut) the infimum in the definition of $l_D(z;P;\nu)$ will be attained by some analytic disc defining $l_D(z;A;\nu_{|A})$ for some $\emptyset\neq A\subset P$.

The function $l_D(\cdot;P;\nu)$ is called {\it the Lempert function with the poles at $P$ and with the weight function $\nu$ (or weights $\nu_j$)}.

Analoguously we define {\it the pluricomplex Green function $g_D(z;P;\nu)$ with the poles at $P$ and the weight function $\nu$} as the supremum of numbers $u(z)$ 
over all negative plurisubharmonic functions $u:D\to[-\infty,0)$ with logarithmic poles at $P$, i.e. such that
$$
u(\cdot)-\nu_j\log||\cdot -p_j||
$$
is bounded above near $p_j$, $j=1,\ldots,N$.

It is trivial that $g_D(z;P;\nu)\leq l_D(z;P;\nu)$. D. Coman conjectured in \cite{Com 2000} the equality $l_D(\cdot;P;\nu)=g_D(\cdot;P;\nu)$ for all convex domains $D$. 

The conjecture has an obvious motivation in the Lempert Theorem (see \cite{Lem 1981}) which implies the equality in the case $N=1$, and in the fact that the equality in the case of the unit ball and two poles with equal weights ($D=\mathbb B_n$, $N=2$, $\nu_1=\nu_2$)  holds (see \cite{Com 2000} and also \cite{Edi-Zwo 1998}).

The conjecture turned out to be false. The first counterexample was found in \cite{Car-Wie 2003} ($D:=\mathbb D^2$, $N=2$ and different weights).
Later a counterexample was found in the case of the bidisk ($D=\mathbb D^2$) with $N=4$ and all weights equal 
(see \cite{Tho-Trao 2003}). The conjecture does not hold either for $N=3$ and the weights equal in the case of the bidisc, or any bounded domain
 (see \cite{Tho 2012}).

The simplest non-trivial case that was not clear yet was the case of the bidisc, two poles and equal weights. Recall that a partial positive answer in this case was found in \cite{Car 1999} 
(see also \cite{Edi-Zwo 1998}) in the case the poles were lying on $\mathbb D\times\{0\}$. In \cite{Wik 2003} numerical computations were carried out which strongly suggested that the equality in the case $D=\mathbb D^2$, $N=2$, $\nu_1=\nu_2$ should hold. The aim of this paper is to show that actually the Coman conjecture holds in the bidisk ($D=\mathbb D^2$), $N=2$, two arbitrary poles and $\nu_1=\nu_2$. 
In our proof we show even more: the equality of the Carath\'eodory function (defined below) and the Lempert function with two poles and equal weights in the bidisc. The methods we use originated with the
study of the Nevanlinna-Pick problem for the bidisc.

\section{Nevanlinna-Pick problem, $m$-complex geodesics, formulation of the solution}
As already mentioned, the aim of the paper is to show a more general result than one claimed in the Coman conjecture for the bidisc, two poles and equal weights. 
To formulate the main result we need to introduce a new function. Since we shall be interested in equal weights we restrict
ourselves from now on to the case when $\nu\equiv 1$. To make the presentation clearer we adopt the notation $d_D(z,\{p_1,\ldots,p_N\}):=d_D(z;\{(p_1,1),\ldots,(p_N,\nu_N)\})$ ($d=l$ or $g$) where $p_j\in D$'s are pairwise disjoint, $j=1,\ldots,N$.

Let us recall the definition of {\it the Carath\'eodory function with the poles at $p_j$} (\textit{with weights equal to one})
\begin{multline}
c_D(z,p_1,\ldots,p_N):=\\
\sup\{\log |F(z)|:F\in\mathcal O(D,\mathbb D), F(p_j)=0,\; j=1,\ldots,N\}.
\end{multline}
It is simple to see that 
$$
c_D(\cdot,p_1,\ldots,p_N)\leq g_D(\cdot,p_1,\ldots,p_N)\leq l_D(\cdot,p_1,\ldots,p_N).
$$

Our main result is the following.
\begin{theorem}\label{thm:main-theorem}
Let $p,q\in\mathbb D^2$ be two distinct points. Then
$$
c_{\mathbb D^2}(z;p;q)=l_{\mathbb D^2}(z;p;q),\; z\in\mathbb D^2.
$$
\end{theorem}
Note that the function $F$ for which the supremum in the definition of the Carath\'eodory function is attained always exists. 
On the other hand in the case where
$D$ is a taut domain, for a point $z\in D$ and pole set $P$ there are always a set $\emptyset\neq Q=\{q_1,\ldots,q_M\}\subset P$ 
and a mapping $f\in\mathcal O(\mathbb D,D)$, $\lambda_j\in\mathbb D$ such that $f(0)=z$, $f(\lambda_j)=q_j$, $j=1,\ldots,M$ and
$l_D(z;P)=l_D(z;Q)=\sum\sb{j=1}\sp{M}\log|\lambda_j|$. Consequently, in case the equality $c_D(z;p_1,\ldots,p_N)=l_D(z;p_1;\ldots;p_N)$ holds,
 there exist $f\in\mathcal O(\mathbb D,D)$, $F\in\mathcal O(D,\mathbb D)$ such that $f(0)=z$, $f(\lambda_j)=q_j$, $F(q_j)=0$, $|F(0)|=\prod_{j=1}^M|\lambda_j|$, $j=1,\ldots,M$, and (thus) $F\circ f$ is a finite Blaschke product of degree $M\leq N$. 
This observation leads us to introduce and consider the notions of $m$-extremals and $m$-geodesics.

First recall that given a system of $m$ pairwise different numbers $(\lambda_1,\ldots,\lambda_m)$, $\lambda_j\in\mathbb D$, a domain $D\subset\mathbb C^n$, a holomorphic mapping $f:\mathbb D\to D$ is called {\it a (weak) $m$-extremal} for $(\lambda_1,\ldots,\lambda_m)$ if there is no holomorphic mapping $g:\mathbb D\to D$ such that $g(\mathbb D)\subset\subset D$ and $g(\lambda_j)=f(\lambda_j)$, $j=1,\ldots,m$. In case $f$ is $m$-extremal with respect to any choice of $m$ pairwise different arguments the mapping $f$ is called \textit{$m$-extremal}. A holomorphic mapping $f:\mathbb D\to D$ is called {\it an $m$-geodesic} if there is an $F\in\mathcal O(D,\mathbb D)$ such that $F\circ f$ is a finite Blaschke product of degree at most $m-1$. The function $F$ will be called {\it the left inverse to $f$}. It is immediate to see that any $m$-geodesic is an $m$-extremal.

The notions of (weak) $m$-extremals and $m$-geodesics, which have clear origin in Nevanlinna-Pick problems for functions in the unit disk, have been recently introduced and studied in \cite{Agl-Lyk-You 2013}, \cite{Agl-Lyk-You 2014}, \cite{Kos-Zwo 2014}, \cite{Kos 2014} and \cite{War 2014}. It is worth recalling that the description of $m$-extremals in the unit disc is classical and well-known. The mapping $h\in\mathcal O(\mathbb D,\mathbb D)$ is $m$-extremal for $(\lambda_1,\ldots,\lambda_m)$, $\lambda_j\in\mathbb D$ if and only if $h$ is a finite Blaschke product of degree less than or equal to $m-1$. Moreover, in such a case the interpolating function is uniquely determined (see \cite{Pick 1916}).

The remark after Theorem \ref{thm:main-theorem} on the form of functions for which the extremum in the definition of the Lempert function may be attained may be formulated as follows. 
For any taut domain $D$, for any system of poles $P=\{p_1,\ldots,p_N\}\subset D$ and any $z\in D\setminus P$ there are a subset $Q=\{q_1,\ldots,q_M\}\subset P$ and $f\in\mathcal O(\mathbb D,D)$ such that $f(\lambda_j)=q_j$, $j=1,\ldots,M$, $f(0)=z$, and $f$ is a weak $(M+1)$-extremal for $(0,\lambda_1,\ldots,\lambda_{M})$. Assuming additionally the equality $c_D(z;P)=l_D(z;P)$ would then imply the existence of a special $(M+1)$-geodesic, the one having some subset $Q\subset P$ in its image but such that the left inverse $F$ maps the whole set $P$ to $0$. Consequently a necessary (but not sufficient!) condition for having the desired equality at $z$ for the set of poles $P$ is the existence of some $(M+1)$-geodesic passing through a subset $Q\subset P$ and mapping $0$ to $z$.

Below we present a result on uniqueness of left inverses for $m$-geodesics in convex domains in $\mathbb C^2$ which we shall use in a (very special) case of the bidisk. The result is a simple generalization of a similar result formulated for $2$-geodesics that can be found in \cite{Kos-Zwo 2013} (however, for the clarity of the presentation we restrict ourselves to the dimension two). 
We also present its proof here for the sake of completeness.

\begin{lemma}\label{lemma:uniqueness-of-left-inverses}
Let $D$ be a convex domain in $\mathbb C^2$, $\lambda_j\in\mathbb D$, $j=1,\ldots,m$, $m\geq 2$, be pairwise different and let $f,g:\mathbb D\to D$ be such that $f(\lambda_j)=g(\lambda_j)=:z_j$ and $f\not\equiv g$. Assume additionally that $F,G\in\mathcal O(D,\mathbb D)$ are such that $F\circ f$ and $G\circ g$ are Blaschke products of degree less than or equal to $m-1$. Then $F\equiv G$. Moreover, for any $\mu\in\mathbb C$ and $\lambda\in\mathbb D$ such that $\mu f(\lambda)+(1-\mu)g(\lambda)\in D$ we have the equality
$$
F(\mu f(\lambda)+(1-\mu)g(\lambda))=F(f(\lambda)).
$$
\end{lemma}
\begin{proof}
For $t\in[0,1]$ define $h_t:=tf+(1-t)f\in\mathcal O(\mathbb D,D)$. Then $h_t(\lambda_j)=z_j$, $j=1,\ldots,m$ so, due to the uniqueness of the solution of the extremal problem in the disk, we get that $F\circ h_t\equiv G\circ h_t=:B$, $t\in[0,1]$, is a finite Blaschke product of degree $\leq m-1$. Consequently, we get the equality $F\equiv G$ on the set
$$
\{tf(\lambda)+(1-t)g(\lambda)=g(\lambda)+t(f(\lambda)-g(\lambda)):t\in[0,1],\lambda\in\mathbb D\}.
$$

Let $\emptyset\neq U\subset\subset\mathbb D$ be such that $f(\lambda)\neq g(\lambda)$, $\lambda\in U$, and $g_{|U}$ is injective. Consider the function
$$
\Phi(\mu,\lambda):=g(\lambda)+\mu(f(\lambda)-g(\lambda)),\; (\mu,\lambda)\in\mathbb C\times\mathbb D.
$$
There is a domain $\mathbb C^2\supset \Omega\supset [0,1]\times U$ such that $\Phi(\Omega)\subset D$. The identity principle implies that $F\circ\Phi\equiv G\circ \Phi$ on $\Omega$. 
But $\Phi$ is injective so $\Phi(\Omega)$ is open;
therefore, $F\equiv G$ on $D$. The additional property follows similarly, by the identity principle applied to function $\mu\mapsto F\circ h_{\mu}(\lambda)-G\circ h_{\mu}(\lambda)$, for any $\lambda\in\mathbb D$.
\end{proof}

\section{Properties of extremals for the Lempert function in case the Coman conjecture holds}\label{section:heuristic}

Let us now restrict our considerations to the case of the bidisc and two poles  $p,q\in\mathbb D^2$, $p\neq q$. Without loss of generality we may assume that $z=(0,0)$. Simple continuity properties of the Lempert and Carath\'eodory function allow us to reduce the Coman conjecture to the proof of the equality 
$$
c(p,q):=c_{\mathbb D^2}((0,0),p,q)=l_{\mathbb D^2}((0,0),p,q)=:l(p,q)
$$ 
for $(p,q)$ from some open, dense subset of $\mathbb D^2\times\mathbb D^2\setminus\triangle$ 
to be defined later ($\triangle$ denotes the diagonal in the corresponding Cartesian product $X\times X$, here $X=\mathbb D^2$).

Below we shall make some heuristic (though formal) reasoning which will lead us to the structure of the proof of the equality $c(p,q)=l(p,q)$ presented later.

Our aim will be to show the existence of (special) left inverses for those $3$-extremals (or $2$-extremals) for which the infimum in the definition of $l(p;q)$ is attained.

\begin{remark}\label{remark:inverses}
Note that assuming we have the equality for $(p,q)\in\mathbb D^2\times\mathbb D^2$, $p\neq q$, $p,q\neq(0,0)$ there would be two possibilities (up to a permutation of variables $p$ and $q$)

\begin{enumerate}[(i)]
\item there are holomorphic $\varphi:\DD\to \DD^2$, $F:\DD^2\to \DD$ and $\alpha\in (0,1)$ such that $\varphi(0)=(0,0)$, $\varphi(\alpha) = p$ and $F(p) = F(q) = \alpha$, $F(0,0)=0$.

Then $F(\varphi(\lambda)) = \lambda$, so $\varphi(\lambda) = (\omega \lambda, \psi(\lambda))$ where $|\omega|=1$ (up to switching coordinates). If $\psi\notin\Aut(\DD)$ then Lemma~\ref{lemma:uniqueness-of-left-inverses} implies that $F(z) = \bar \omega z_1$ so $p_1=q_1$ and $|p_2|\leq |p_1|$. Therefore, in this case only a thin (nowhere-dense) set of points $(p,q)$ would be covered and consequently that could be ``neglected''.

The second subcase is when $\psi\in \Aut(\DD)$ and $\psi(0)=0$. But then $|p_1|=|p_2|$ and $q$ may be ``quite general'', when ``quite general'' depends on left inverses to $\lambda\mapsto (\lambda, \lambda)$ in $\DD^2$. But as before the set of $(p,q)$'s involved forms a thin set.

\item The function $\varphi$ realizing the infimum is a weak $3$-extremal with respect to $(0,a,b)$ such that $\varphi(0)=(0,0)$, $\varphi(a) = p$, $\varphi(b)=q$. The (hypothetical) special left inverse 
$F:\DD^2\to \DD$ would satisfy the equalities $F(p)=F(q)=0$ and $F(0)= a b$. Consequently $F\circ \phi = m_a m_b$, where $m_a$, $m_b$ are (idempotent) M\"obius maps. 
We have two possibilities again:

a) $\phi$ is a geodesic ($2$-extremal). This holds if either
\begin{itemize}
\item $|p_2|<|p_1|$, $|q_2|<|q_1|$ and $m(\frac{p_2}{p_1} , \frac{p_2}{p_1} ) \leq m(p_1, p_1)$, or
\item $|p_1|<|p_2|$, $|q_1|<|q_2|$ and $m(\frac{p_1}{p_2} , \frac{q_1}{q_2} ) \leq m(p_2, p_2)$, or
\item $p_2=\omega p_1$ and $q_2=\omega q_1$ for some unimodular $\omega$.
\end{itemize}

b) $\phi$ is not a $2$-extremal. First note that $\phi(\lambda)=\lambda\psi(\lambda)$ where $\psi$ is a $2$-extremal (geodesic) (see e. g. \cite{Kos-Zwo 2014}). Consequently, (up to a permutation of the coordinates) $\phi(\lambda)=\lambda(m(\lambda),h(\lambda))$ where $m$ is some M\"obius map and $h\in\mathcal O(\mathbb D,\mathbb D)$. In the case $h$ is not a M\"obius map the mapping $\phi$ is not uniquely determined -- in the sense that for the triple $(0,a,b)$ there exist also another $3$-extremal mapping $\tilde\phi$ which maps this triple of numbers to the same triple of points. But existence of the left inverse already gives its uniqueness (see Lemma~\ref{lemma:uniqueness-of-left-inverses}); moreover, it follows from the same lemma that $F(\lambda m(\lambda),\mu)=m_a(\lambda)m_b(\lambda)$
for any $\mu\in \mathbb D$, which easily implies that $F(z)=a(z_1)$ where $a$ is some M\"obius map. But the last property may hold only if $p_1=q_1$. In other words, this may hold only for pairs $(p,q)$ from a thin set.

The above considerations suggest that the generic case for $\phi$ being a $3$-extremal from the definition of the Lempert function which are not $2$-extremals should be the one given by the formula 
\begin{equation}\label{equation:three-extremals}
\phi(\lambda)=\lambda(m(\lambda),n(\lambda)), \lambda\in\mathbb D
\end{equation} 
where $m$ and $n$ are M\"obius maps.
\end{enumerate}
\end{remark}

Our aim is now to show what the necessary form of functions $F\in\mathcal O(\mathbb D^2,\mathbb D)$ such that $F\circ f$ is a Blaschke product should be. We present below the reasoning employing some results of McCarthy and Agler. Let us also mention that G. Knese (see \cite{Kne 2014}) let us know about another approach which leads to the same form of left inverses.

Without loss of generality we may assume that the mapping from \eqref{equation:three-extremals} satisfies $m=m_{\alpha}$, $n=n_{\beta}$ where $\alpha,\beta\in\mathbb D$. 
We are looking for a necessary form of a function $F:\DD^2 \to \DD$ such that $F\circ \varphi = B$, where $B$ is a Blaschke product of degree $2$.
Since $F\circ \varphi(0)=0$, it suffices to consider the case when $B(\lambda)= \lambda m(\lambda)$ for some M\"obius map $m$.

We have the following situation: 
$$F(\lambda m_\alpha (\lambda), \lambda m_\beta(\lambda)) = \lambda m_\gamma (\lambda)$$ 
and we are looking for a formula for $F$. We consider only the case when $\alpha \neq \beta$ (as $\alpha=\beta$ gives points from a thin set). Note that it may happen that $\gamma = \alpha$ or $\beta=\gamma$ and then $F$ depends just on one variable (use Lemma~\ref{lemma:uniqueness-of-left-inverses}).

Otherwise, assuming that $F$ and $\gamma$ do exist consider the following Pick problem: 
$$
\begin{cases}(0,0)\mapsto 0\\ (\gamma m_\alpha(\gamma) , \gamma m_\beta (\gamma)) \mapsto 0,\\ (\lambda' m_\alpha (\lambda'), \lambda' m_\beta (\lambda'))\mapsto \lambda' m_\gamma (\lambda'),
\end{cases}
$$ where $\lambda'$ is any point in $\DD$. It is quite clear that this problem is strictly
$2$-dimensional, extremal and non-degenerate (with the the notions understood as defined in 
\cite[Chapter 12]{Agl-McC 2002}, itself drawing from \cite{Agl-McC 2000} where the terminology is slightly different). 
Therefore, it follows from \cite[Theorem 12.13, p. 201--204]{Agl-McC 2002} that the above problem has a unique solution which is given by a rational inner function of degree $2$, with no
terms in $x_1^2$ or $x_2^2$. 
It is easily seen that the solution to this problem is a left inverse we are looking for. Therefore, 
$$F(x)= \frac{A x_1 + B x_2 +  C x_1 x_2}{1+Dx_1+Ex_2+G x_1 x_2}.$$ 
%$$F(x)= \frac{A x_1 + B x_2 + Cx_1^2 + D x_1 x_2 + E x_2^2}{\cdots}.$$ 
Now we proceed in a standard way: comparing multiplicities in the poles of $m_\alpha$ and $m_\beta$, etc.
% we see that $C=E=0$ etc. 
After additional calculations we get that $A+B=1$ and then 
$$F(x) = \frac{t x_1 + (1-t) x_2 - \omega x_1 x_2}{1- ((1-t) x_1 + t x_2) \omega},$$ where $t\in (0,1)$ and $\omega\in \TT$ depends only on $\alpha$ and $\beta$. In particular, $\gamma$ is a convex combination of $\alpha$ and $\beta$.

The above considerations make us present a formal result that we shall use in the sequel. 

\begin{lemma}\label{lemma:left-inverses}
Let $\alpha,\beta\in\mathbb D$, $\alpha\neq\beta$, $t\in[0,1]$, $\omega,\tau\in\mathbb T$. Define $\varphi(\lambda):=\lambda(\omega m_{\alpha}(\lambda),m_{\beta}(\lambda))$
and let
\begin{equation}
 F(x):=\frac{t\bar\omega x_1+(1-t)x_2+\tau\bar\omega x_1x_2}{1+\tau((1-t)\bar\omega x_1+tx_2)},\;x=(x_1,x_2)\in\mathbb D^2.
\end{equation}
Set $F(\varphi(\lambda))=:\lambda f(\lambda)$, $\lambda\in\DD$. Denote $f(0)=\gamma:=t\alpha+(1-t)\beta$. Then $f$ is an automorphism of $\mathbb D$ (equal to $m_{\gamma}$) 
if and only if
$\tau=\frac{\overline{\alpha-\beta}}{\alpha-\beta}$.
\end{lemma}
\begin{proof}
The proof of the above lemma reduces to showing that in the inequality $|f^{\prime}(0)|/(1-|f(0)|^2)\leq 1$ the equality holds if and only if $\tau=\frac{\overline{\alpha-\beta}}{\alpha-\beta}$ which is elementary although tedious.
\end{proof}
A direct consequence of the above result is the following.

\begin{corollary} 
\label{coreq}
Let $\alpha,\beta,t,\omega,\gamma,\phi$ be as above,
and $\tau=\frac{\overline{\alpha-\beta}}{\alpha-\beta}$. Then we have the equality 
\begin{equation}
 c(\varphi(\lambda),\varphi(m_{\gamma}(\lambda)))=l(\varphi(\lambda),\varphi(m_{\gamma}(\lambda))),\;\lambda\in\DD,\;m_{\gamma}(\lambda)\neq\lambda.
\end{equation}
\end{corollary}
The above equality is a key one. It will turn out that the set of pairs of points $(\phi(\lambda),\phi(m_{\gamma}(\lambda))$ (parametrized by $(\alpha,\beta,c,t,\omega)$) will build an open set, which together with the one constructed with the help of extremals for the Lempert functions being $2$-geodesics will be dense in $\mathbb D^2\times\mathbb D^2$ -- that will complete the proof.

%Consequently, the extremals that should be considered in this case are of the following form
%$\phi(\lambda) = (\lambda m(\lambda), \lambda n(\lambda))$ for some M\"obius maps $m$, $n$.

\section{Proof of the equality $c(p;q)=l(p;q)$;  'one-dimensional' case}
After the preliminary remarks in the previous section we start with the formal proof. We shall consider two open sets in $\mathbb D^2\times\mathbb D^2\setminus\triangle$ 
whose union forms a dense subset of $\mathbb D^2\times\mathbb D^2\setminus\triangle$ and on each part the desired equality will be proven. Let us also denote 
$\sigma(p,q):=((p_2,p_1),(q_2,q_1))$, $p,q\in\mathbb D^2$. The equality of the Lempert and Carath\'eodory function
on the sets will be simple and will actually follow from the earlier reasoning. The key problem will be with the proof of the density of the union of the two sets considered. 
In this section we deal with the set defined with the help of $2$-geodesics.

Define $U$ as the set of points $(p;q)\in\mathbb D^2\times\mathbb D^2$ satisfying the following inequalities
\begin{equation}
|p_2|<|p_1|,|q_2|<|q_1| \text{ and } m\left(p_2/p_1,q_2/q_1\right)<m(p_1,q_1),
\end{equation}
where $m$ is the M\"obius distance on the unit disc given by the formula $m(\lambda_1,\lambda_2):=\left|\frac{\lambda_1-\lambda_2}{1-\bar\lambda_1\lambda_2}\right|$.

Denote $\Omega_1:=U\cup \sigma(U)$.

Let $(p,q)\in\Omega_1$. Without loss of generality we may assume that $(p,q)\in U$. Let $\phi(\lambda):=(\lambda,\lambda\psi(\lambda))$ where $\psi\in\mathcal O(\mathbb D,\mathbb D)$ is such that $\psi(p_1)=p_2/p_1$, $\psi(q_1)=q_2/q_1$. Let $F(z):=m_{p_1}(z_1)m_{q_1}(z_1)$, $z\in\mathbb D^2$. Then $\phi(0)=(0,0)$, $\phi(p_1)=p$, $\phi(q_1)=q$, $F(0,0)=p_1q_1$ and $F(p)=F(q)=0$ which gives the equality
$$
c(p;q)\leq l(p;q)\leq \log|p_1q_1|\leq c(p;q),
$$
from which the desired equality holds on $U$ (and thus on $\Omega_1$).

\section{Proof of the equality $c(p;q)=l(p;q)$; 'two-dimensional' case}
As announced earlier we consider now the set given by $3$-geodesics that are not $2$-geodesics and that appeared in our heuristic reasoning in Section~\ref{section:heuristic}. The equality of the Lempert and Carath\'eodory
functions 
%will be simple. 
follows from Corollary \ref{coreq}.

The remaining proof of the density of the union of sets introduced in the present and previous sections will constitute the main part of this section.

Consider a real-analytic mapping
$$\Phi:\DD\times \DD \times \DD \times \TT \times (0,1) \to \DD^2\times\DD^2$$ 
given by the formula (below and in the sequel $\gamma:=t\alpha+(1-t)\beta$)
$$
(\alpha, \beta, c, \omega, t) \mapsto \left(\varphi_{\alpha, \beta, \omega} (c ), 
\varphi_{\alpha, \beta, \omega} ( m_\gamma(c) )  \right),
$$
where $\varphi_{\alpha, \beta, \omega} (\zeta) := (\omega \zeta m_\alpha(\zeta), \zeta m_\beta(\zeta) )$.

%$$(\alpha, \beta, c, \omega, t) \mapsto (\omega c m_\alpha(c), c m_\beta(c),\omega m_\gamma(c) m_\alpha(m_\gamma (c)), m_\gamma(c) m_\beta(m_\gamma(c)))$$

%Therefore, the equality is equivalent to the fact that the image of $\Phi$ contains $D$. 

Motivated by the considerations in Section~\ref{section:heuristic} we define open sets.

Denote $\mathcal A:=\{(p,q)\in\DD^2\times\DD^2:p_1=q_1 \text{ or } p_2=q_2\}$ and 
\begin{equation}
F_1:=\{(p,q)\in\DD^2\times\DD^2: |p_2|> |p_1|,\ |q_2|<|q_1|\}.
\end{equation}
We also define the set $F_2$ as the set of points
$(p,q)\in\DD^2\times\DD^2$
satisfying the following inequalities
\begin{equation}
|p_2|< |p_1|,\ |q_2|<|q_1|, m\left(\frac{p_2}{p_1}, \frac{q_2}{q_1}\right) > m(p_1,q_1).
\end{equation}

Let $F_3=\sigma (F_1)$, and $F_4=\sigma(F_2)$. Let $E_j:=F_j\setminus \mathcal A$.

%The continuity properties of the functions involved allow us to reduce the problem of proving the equality in the remaining cases to the proof of the equality in 
%the open set  $D:=E_1\cup E_2\cup E_3\cup E_4$.

Define
  $$
\Omega_2:=E_1\cup E_2\cup E_3\cup E_4.
$$

Certainly the sets $E_j$ are disjoint and open. Moreover, they are connected. Actually, $\mathcal A$ is an analytic set so it is sufficient to show the connectivity of $F_j$. But 
$F_1$ is the image of $\DD\times\DD_*\times\DD_*\times\DD$ under the mapping $ \lambda\mapsto (\lambda_1\lambda_2,\lambda_2,\lambda_4,\lambda_3\lambda_4)$. On the other hand the set $F_2$ is the image, under the mapping $\lambda\mapsto (\lambda_1,\lambda_1\lambda_2,\lambda_3,\lambda_3\lambda_4)$ of the set $B:=\{\lambda\in \DD\times\DD_*\times\DD\times\DD_*:
m(\lambda_1,\lambda_3)<m(\lambda_2,\lambda_4)\}$. The last set is connected because any point $\lambda\in B$ may be connected by the curve $[0,1]\owns t\mapsto(t\lambda_1,\lambda_2,t\lambda_3,\lambda_4)$ with $(0,\lambda_2,0,\lambda_4)$. And now it is sufficient to see that the set
$\{0\}\times\DD_*\times\{0\}\times\DD_*$ is arc-connected.

Let $G_j:=\Phi^{-1}(E_j)$. To finish the proof of the assertion it suffices to show that $$\Phi|_{G_j}:G_j\to E_j$$ is surjective. In fact, in such a case $\Phi(G_j)=E_j$ so the equality $l=c$ holds on $\Omega_2$, which together with $\Omega_1$ builds a dense subset of $\mathbb D^2\times\mathbb D^2\setminus\triangle$.

\bigskip

Therefore, to finish the proof of the theorem we go to the proof of the surjectivity of the mappings defined above.

Without loss of generality we may restrict to the cases $j=1,2$.

First note that the sets $G_j$ are non-empty. Therefore, to finish the proof it is sufficient to show that $\Phi(G_j)$ is open and closed in $E_j$.

First we show that $\Phi(G_j)$ is closed. The proof may be conducted with the standard sequence procedure; however, we shall make use of descriptions of left inverses that were given in Section~\ref{section:heuristic}.

Take $(p,q)$ in the closure of $\Phi(G_j)$ with respect to $E_j$. The continuity property implies that $c(p,q)=l(p,q)$ with the left inverse $F$ in the same form as in Lemma~\ref{lemma:left-inverses}. It follows from Remark~\ref{remark:inverses} that a function for which the supremum in the definition of $c(p,q)$ is attained is as in Lemma~\ref{lemma:left-inverses}. But the extremal mapping in the definition of the Lempert function must be of the form as in (\ref{equation:three-extremals}). This, together with the uniqueness part of Lemma~\ref{lemma:left-inverses} gives that $(p,q)\in\Phi(G_j)$.

To show that the image is open it suffices to prove that $\Phi$ is locally injective.

So assume that $\Phi(\alpha, \beta, c, \omega, t) = \Phi(\tilde \alpha, \tilde\beta,\tilde c,\tilde \omega, \tilde t)$.

Let $ \varphi:=\varphi_{\alpha,\beta,\omega}$, $\tilde\varphi:=\varphi_{\tilde\alpha,\tilde\beta,\tilde\omega}$.

%(\lambda) = (\omega \lambda m_\alpha(\lambda), \lambda m_\beta (\lambda))$ and $\tilde \varphi(\lambda) = (\tilde \omega \lambda m_{\tilde \alpha}(\lambda), \lambda m_{\tilde \beta} (\lambda))$. 

Let $F(x) = \frac{\bar \omega  t x_1 + (1-t) x_2 + \eta x_1 x_2}{1 - ((1-t) \bar \omega x_1 + t x_2)\eta},$ where $\eta$ is properly chosen. It simply follows from the previous discussion that $F$ is a left inverse to both $\phi$ and $\tilde \phi$. Therefore, $F=\tilde F$ ($\tilde F$ denotes the appropriate left inverse to $\tilde\varphi$) so $t=\tilde t$ and $\omega = \tilde \omega$. 
Moreover, $c m_\gamma(c)= \tilde c m_{\tilde \gamma} (\tilde c)=:l\neq 0$. Therefore, it suffices to show the local injectivity of the function
$$\Psi: (\alpha, \beta, c)\mapsto\left(c m_\alpha(c), \frac lc m_\alpha (l/c), c m_\beta(c), \frac lc m_\beta (l/c)\right)$$ defined for $(\alpha,\beta, c)\in \DD^3$ such that $(z,w)=\Phi(\alpha, \beta, c)$ satisfies $|z_1|\neq |z_2|$, $|w_1|\neq |w_2|$, $z_1\neq w_1$ and $z_2\neq w_2$ (in particular, $\alpha\neq\beta$, $c\neq 0$).

\begin{prop}
$\Psi$ is locally injective. Moreover, $\Psi$ is two-to-one.
\end{prop}
\begin{proof}
Observe first that $\Psi(\alpha, \beta, c) = \Psi(-\alpha, -\beta, -c)$. Therefore, to get the assertion it suffices to show that for fixed points $z:=(z_1,z_2)$, $w:=(w_1,w_2)$ such that $z_1\neq z_2$, $w_1\neq w_2$, $z_1\neq w_1$ and $z_2\neq w_2$ the equation $\Phi(\alpha, \beta, c) = (z_1, z_2, w_1, w_2)$ has at most two solutions.

From the equation we deduce that
\begin{align*}
\alpha =c \frac{z_2(1-z_1/l)}{z_2 - z_1} + \frac 1c \frac{z_1(z_2-l)}{z_2-z_1},\ &\text{ and }\
\bar \alpha =  c \frac{1 - z_2 /l}{z_1 - z_2} + \frac 1c \frac{z_1 - l}{z_1 - z_2}\\
\beta =c \frac{w_2(1-w_1/l)}{w_2 - w_1} + \frac 1c \frac{w_1(w_2-l)}{w_2-w_1},\ \ &\text{ and }\
\bar \beta =  c \frac{1 - w_2 /l}{w_1 - w_2} + \frac 1c \frac{w_1 - l}{w_1 - w_2}.
\end{align*}
We can write the above equations in the form 
\begin{equation*}
\begin{pmatrix} \alpha \\ \beta \end{pmatrix}= M \begin{pmatrix} c \\ \frac1c \end{pmatrix} ,
\quad
\begin{pmatrix} \bar \alpha \\ \bar \beta \end{pmatrix}= N \begin{pmatrix} c \\ \frac1c \end{pmatrix},
\end{equation*}
where $M, N \in \mathbb C^{2\times 2}$. Set $v:= \begin{pmatrix} c \\ \frac1c \end{pmatrix}$, the
equations imply that $Mv = \bar N \bar v$. 

Notice that 
\begin{align*}
\det M &= \frac{z_2(1- z_1/l) w_1 (w_2-l) - w_2(1- w_1/l) z_1 (z_2-l)}{(z_2 - z_1)(w_2 - w_1)},
\\
\det N &= \frac{(1-z_2/l) (w_1 - l) - (1- w_2/l) (z_1 - l)}{(z_2 - z_1)(w_2 - w_1)}.
\end{align*}
The hypotheses made on $z$ and $w$ ensure that $(1-z_2/l) (w_1 - l)$ and $(1- w_2/l) (z_1 - l)$
cannot vanish simultaneously, so if $\det N=0$, we see that the equation $\det M=0$ reduces
to $z_2w_1-z_1w_2=0$. Since $l\neq 0$, this together with $\det N=0$ would imply $z_1=z_2$, which
is excluded.  Therefore at least one of the matrices $M$ or $N$ is invertible.  Suppose for now
that $M$ is invertible, we have $v= P \bar v$, with $P:= M^{-1}\bar N$.  Since $\bar v= \bar P v$,
 we see that $v = P \bar P v$.  
 
 Since $M^{-1}$ is invertible, $P=0$ if and only if $N=0$,
 which is impossible by the  hypotheses made on $z$ and $w$.  So $\dim \ker (I-P \bar P) \le 1$,
 which means, since $v$ cannot be $0$,
  that there is a nonzero vector $w\in \mathbb C^2$, depending only on $z,w,l$, such that
 $v$ is colinear to $w$, which implies $c^2=w_1/w_2$. So we have at most two possible values 
 for $(\alpha, \beta, c)$.
 
 If $\det M=0$, then $N$ is invertible and we reason in the same way starting from $v = N^{-1}\bar M v$.

\end{proof}

\end{document}